\documentclass[12pt]{amsart}
\usepackage{amsmath,amsthm,amsfonts,amssymb,latexsym,enumerate,xcolor,comment}
\usepackage[colorlinks,pagebackref]{hyperref}
\usepackage[english]{babel}

\newcommand{\ZZ}{{\mathbb{Z}}}

\newcommand{\bC}{{\mathbf C}}

\newcommand{\SSS}{\mathsf{S}}
\newcommand{\AAA}{\mathsf{A}}

\renewcommand{\exp}{{{\operatorname{exp}}}}

\newcommand{\Irr}{{{\operatorname{Irr}}}}

\newcommand{\fpr}{\operatorname{fpr}}
\newcommand{\fix}{\operatorname{fix}}

\newtheorem{thm}{Theorem}[section]
\newtheorem{lem}[thm]{Lemma}

\newtheorem{prop}[thm]{Proposition}

\newtheorem*{conA'}{Conjecture A'}

\theoremstyle{definition}
\newtheorem{rem}[thm]{Remark}
\newtheorem{defn}[thm]{Definition}
\newtheorem{notation}[thm]{Notation}

\numberwithin{equation}{section}

\begin{document}

\title[Wreath products and the non-coprime $k(GV)$ problem]{Wreath products and the non-coprime\\ $k(GV)$ problem}

\author[N. N. Hung]{Nguyen N. Hung}
\address{Department of Mathematics, The University of Akron, Akron,
OH 44325, USA}
\email{hungnguyen@uakron.edu}

\author[A. Mar\'oti]{Attila Mar\'oti}
\address{Hun-Ren Alfr\'ed R\'enyi Institute of Mathematics, Re\'altanoda Utca 13-15, H-1053, Budapest, Hungary}
\email{maroti@renyi.hu}

\author[J. Mart\'{\i}nez Madrid]{Juan Mart\'{\i}nez Madrid}
\address{Departament de Matem\`atiques, Universitat de Val\`encia, 46100
  Burjassot, Val\`encia, Spain}
\email{Juan.Martinez-Madrid@uv.es}

\thanks{The first author gratefully acknowledges the support of the AMS-Simons Research Enhancement Grant (AWD-000167 AMS)
and the UA Faculty Research Grant (FRG 1747). The second author was
supported by the National Research, Development and Innovation
Office (NKFIH) Grant No.~K138596, No.~K132951 and Grant No.~K138828.
The research of the third author was supported by Ministerio de
Ciencia e Innovaci\'on (Grant PID2022-137612NB-I00 funded by
MCIN/AEI/10.13039/501100011033 and ``ERDF A way of making Europe"),
and by Generalitat Valenciana CIAICO/2021/163 and CIACIF/2021/228.
Finally, we are grateful to the referee for several helpful comments
that significantly improved the exposition of the paper.}




\begin{abstract}
Let $G = X \wr H$ be the wreath product of a nontrivial finite group
$X$ with $k$ conjugacy classes and a transitive permutation group
$H$ of degree $n$ acting on the set of $n$ direct factors of $X^n$.
If $H$ is semiprimitive, then $k(G) \leq k^n$ for every sufficiently
large $n$ or $k$. This result solves a case of the non-coprime
$k(GV)$ problem and provides an affirmative answer to a question of
Garzoni and Gill for semiprimitive permutation groups. The proof
does not require the classification of finite simple groups.
\end{abstract}
\maketitle



\section{Introduction}

Let $G$ be a finite group. Let $k(G)$ be the number of conjugacy
classes of $G$. This is equal to the number of complex irreducible
representations of $G$. Bounding $k(G)$ is a classical problem with
numerous applications in both group theory and representation
theory. There are many results providing upper bounds for $k(G)$.
The most notable one is the $k(GV)$ theorem, which states that
$k(GV)\leq |V|$, where $V$ is an elementary abelian group which is a finite and faithful $G$-module for a
finite group $G$ of order coprime to $|V|$ (see \cite{S}).

In \cite[Problem~1.1]{GT}, Guralnick and Tiep put forward the
non-coprime $k(GV)$ problem:  without assuming the coprime
condition, can one show that $k(GV)\leq|V|$? More precisely, can one
characterize all finite groups $GV$ such that $k(GV) > |V|$? There
are many works on this problem. See \cite{GT}, \cite{Keller1},
\cite[Chapter 13]{S}, \cite{Keller2}, \cite{GuMa}, \cite{Nav},
\cite{FuGu}.

In case $V$ is a (finite, faithful and) irreducible $G$-module for a
finite group $G$ (not necessarily of coprime order to $|V|$) the
semidirect product $GV$ is an affine primitive permutation group
with socle $V$ and degree $|V|$. On the other hand, when $H$ is a
primitive permutation group with non-abelian socle and of degree
$n$, Garzoni and Gill \cite{GG} proved that either $k(H) < n/2$ and
$k(H) = o(n)$ as $n \to \infty$, or $H$ belongs to explicit families
of examples.

In this paper we are interested in bounding $k(G)$ for $G = X \wr
H$, the wreath product of a finite group $X$ and a permutation group
$H$ of degree $n$ acting on the set of direct factors of $X^n$. This
is a case of the non-coprime $k(GV)$ problem when $X$ is an
elementary abelian group or, more generally, when $X =KW$ for some
finite $K$-module $W$ of a finite group $K$. The problem of bounding
$k(X \wr H)$ is also related to \cite{GG} as described below.

Let $k:=k(X)$. Schmid \cite[Proposition 8.5d]{S} proved that if $H$
is a cyclic group of order $n$ (acting regularly), then $k(G) =
(k^{n}-k)/n +kn$ when $n$ is a prime and $k(G) \leq k^{n}-k+kn$ in
general. More recently, Garzoni and Gill \cite[Lemma 4.3]{GG} showed
that if $H$ is a regular permutation group, then $k(G) =
\frac{k^n}{n} + O(n k^{n/2})$. Moreover, they asked \cite[Question
2]{GG} whether $k(G) = O(k^{n})$ for any transitive permutation
group $H$ of degree $n$.

If $n$ is even, $X = C_2$, $k = 2$ and $H = {(C_2)}^{n/2}$ such that
$G = X \wr H$ is the direct product of $n/2$ copies of $C_{2} \wr
C_{2}$, then $k(G) = 5^{n/2} > k^{n}$. Furthermore, if $H =
{(C_2)}^{n/2}$ is replaced by $H = C_{2} \wr C_{n/2}$, then $H$ acts
transitively on the set of factors of $X^{n}$ and $k(G) \geq
5^{n/2}/(n/2) > k^{n}$ for every sufficiently large $n$. This
answers \cite[Question 2]{GG} in the negative (for $k=2$).

On the other hand, we provide an affirmative answer to the
Garzoni-Gill question in the case where $H$ is a primitive
permutation group. In fact, our result extends to the broader class
of transitive groups called \emph{semiprimitive} permutation groups,
defined as those permutation groups in which every normal subgroup
is transitive or semiregular. (A permutation group is called {\it
semiregular} if the stabilizer of any point is trivial.)

\begin{thm}
\label{t2} Let $G = X \wr H$ where $X$ is a nontrivial finite group
with $k$ conjugacy classes and $H$ is a transitive permutation group
of degree $n$ acting on the set of $n$ direct factors of $X^n$. If
$H$ is semiprimitive, then $k(G) \leq k^n$ for every sufficiently
large $n$ or $k$.
\end{thm}

Our proof of Theorem \ref{t2} does not use the classification of
finite simple groups.

In several cases we obtain an asymptotic formula for $k(G)$. For
example, when $H$ is an arbitrary transitive group of order at most
$2^{\sqrt{n}/4}$ (see Theorem \ref{t3}) or when $H$ is primitive
with known exceptions (see Theorem \ref{t1}), then $k(G) = (1+o(1))
(k^{n} / |H|)$ as $n \to \infty$ or $k \to \infty$. It is not true
in general that $k(G) = O(k^{n}/|H|)$; for instance, when $n \to
\infty$ and $H=\SSS_n$, we have $k^n/|H|<1$, yet $k(G)>k(H)$ remains
large. On the other hand, $k(G) \geq k^{n}/|H|$ even for an
arbitrary permutation group $H$ (see Lemma~\ref{l1}).

An important class of primitive permutation groups relevant to our
proof is the so-called \emph{large base} groups (see
Definition~\ref{def}). In Section~\ref{sec:3}, we establish the
asymptotic formula for the case where $H$ is not one of these large
base groups. The specific cases $H\in\{\AAA_n,\SSS_n\}$ are
addressed in Section~\ref{sec:4}, while the remaining large base
groups are treated in Sections~\ref{sec:bounding-orbit} and
\ref{sec:6} using a different approach. This completes the proof of
Theorem~\ref{t2} for primitive $H$. Finally, in
Section~\ref{sec:semi}, we extend the machinery developed in the
earlier sections to prove the result for all semiprimitive groups.


\section{An asymptotic formula for $k(G)$}
\label{sec:3}

Let $G$, $X$, $H$, $k$, and $n$ be as in the statement of Theorem
\ref{t2}. For the moment assume that $H$ is an arbitrary permutation
group of degree $n$. This group $H$ has a natural action on
$\Irr(X^n)=\Irr(X)^n$, the set of complex irreducible characters of
$X^{n}$. Let $\chi_{1}, \ldots , \chi_{f}$ be a list of
representatives of the distinct orbits of $H$ on
$\mathrm{Irr}(X^{n})$. Let $I_{H}(\chi)$ denote the inertia group in
$H$ of a character $\chi$ in $\Irr(X^{n})$.

\begin{lem}
\label{l1}
We have $k(G) = \sum_{i=1}^{f} k(I_{H}(\chi_{i}))$.
\end{lem}

\begin{proof}
Fix an index $i$. The character $\chi_i$ of $X^{n}$ may be extended
to its inertia group $I_G(\chi_i)=X \wr I_{H}(\chi_{i})$ in $G$ by
\cite[p. 154]{JK}. Thus the number of irreducible characters of
$I_G(\chi_i)$ lying above $\chi_i$ is $k(I_{H}(\chi_{1}))$, by
Gallagher's theorem \cite[Corollary 6.17]{I}. The identity now
follows from Clifford's correspondence \cite[Theorem 6.11]{I}.
\end{proof}

For an element $h$ of $H \leq \SSS_n$, let $\sigma(h)$ be the number
of cycles in the disjoint cycle decomposition of $h$. For a finite
group $L$ acting on a finite set $\Omega$, we denote the number of
orbits of $L$ on $\Omega$ by $n(L, \Omega)$.

\begin{lem}
\label{l2}
We have $$n(H, \Irr(X^{n})) = \frac{1}{|H|}\sum_{h\in H}k^{\sigma(h)}.$$
\end{lem}

\begin{proof}
Let $h$ be an arbitrary element of $H$. The number of characters in
$\Irr(X^n)$ fixed by $h$ is $k^{\sigma(h)}$. The statement follows
from the orbit-counting lemma.
\end{proof}

Let $\Delta$ denote the union of all non-regular orbits of $H$
acting on $\Irr(X^{n})$. By the orbit-counting lemma, the number of
non-regular orbits is
\[
\frac{1}{|H|}\sum_{h\in H}
|\fix(h,\Delta)|=\frac{|\Delta|}{|H|}+\frac{1}{|H|}\sum_{h\in
H\backslash \{1\}} |\fix(h,\Delta)|,
\]
where $\fix(h,\Delta)$ is the set of the fixed points of $h$ on
$\Delta$. Let \[\alpha(H) := \max_{1 \not= h \in H} \sigma(h)/n.\]
Then we get
$$|\Delta|\leq\sum_{h\in H\setminus \{1\}}k^{\alpha(H) n}\leq
(|H|-1)k^{\alpha(H) n}.$$ Therefore, the number of non-regular
orbits is at most $$\frac{(|H|-1)k^{\alpha(H) n} +
(|H|-1)k^{\alpha(H) n}}{|H|} < 2 k^{\alpha(H) n}.$$

Now, up to reordering, let $\chi_1,...,\chi_t$ be representatives of
the non-regular orbits. In particular, $t\leq 2 k^{\alpha(H) n}$.
Moreover, by Lemma~\ref{l1}, we obtain
\[
k(G)=\frac{k^n-|\Delta|}{|H|}+\sum_{i=1}^tk(I_H(\chi_i)).
\]
It immediately follows that
\begin{equation}
\label{ee11} k(G) < \frac{k^{n}}{|H|} + 2e k^{\alpha(H) n},
\end{equation}
where $e$ denotes the maximum of $k(T)$ over all the subgroups $T$
of $H$.

If $H$ is regular (or more generally semiregular), then $\alpha(H)
\leq 1/2$, $e \leq n$ and so $k(G) = \frac{k^{n}}{n} + O(n k^{n/2})$
by (\ref{ee11}). This is the bound obtained by Garzoni and Gill
mentioned above.

For any permutation group $H$, as $e \leq 5^{n/3}$ by \cite[Theorem
1.1]{GM}, we have $k(G) = O(k^{n})$, provided that $5^{1/3} <
k^{1-\alpha(H)}$, again by (\ref{ee11}).

\begin{lem}\label{eee111}
We have
\[
k(G) < \Big(1 + \frac{1}{kn}\Big) \frac{k^{n}}{|H|}
\]
if one of the following conditions holds:
\begin{enumerate}
\item $\alpha(H) n \leq n - \log_{k}(2kn|H|^{2})$;

\item $n$ is bounded, $k \to \infty$, and $H$ does not contain a
transposition.
\end{enumerate}
\end{lem}

\begin{proof}
In the first case, we have $(1 - \alpha(H))n \geq
\log_{k}(2kn|H|^{2})$ and so
$$2kne \leq 2kn|H| \leq \frac{k^{(1-\alpha(H))n}}{|H|}.$$ This and (\ref{ee11})
imply the desired inequality. If the second condition holds then
$\alpha(H)n \leq n-2 \leq n - \log_{k}(2kn|H|^{2})$, and the lemma
follows as well. We note that if a primitive permutation group of
degree $n$ contains a transposition then it is $\SSS_{n}$.
\end{proof}

Let $H$ be a transitive permutation group of degree $n$. Let
$\mu(H)$ be the \emph{minimal degree} of $H$. This is the minimal
number of points moved by any nonidentity element of $H$. Let $b(H)$
be the \emph{minimal base size} of $H$. This is the smallest number
of points whose joint stabilizer in $H$ is the identity. We have
$\mu(H) b(H) \geq n$ by \cite[p. 80]{DixMor}. Since $b(H) \leq
\log_{2}|H|$, we obtain $\mu(H) \geq n/ \log_{2}|H|$.

For $h\in H$, let $\fix(h)$ denote the set of fixed points of $h$
and let $\fpr(h) := |\fix(h)|/n$ be the \emph{fixed point ratio} of
$h$. It follows that
\begin{equation}
\label{f1}
\fpr(h) \leq 1 - \frac{1}{\log_{2}|H|}
\end{equation}
for any nonidentity element $h$ in $H$. We have
\begin{equation}
    \label{e3}
    \sigma(h)\leq |\fix(h)|+\frac{n-|\fix(h)|}{2}= \frac{n+|\fix(h)|}{2} = \frac{n}{2}\left(1+ \fpr(h)\right)
\end{equation}
for every $h \in H$. We get $$\alpha(H)n \leq n - \frac{n}{2\log_{2}|H|}$$
by (\ref{f1}) and (\ref{e3}).

The following theorem provides an affirmative answer to
\cite[Question 2]{GG} when $H$ has small order.

\begin{thm}
\label{t3}
If $|H| \leq 2^{\sqrt{n}/4}$, then $k(G) < (1 + \frac{1}{kn}) (k^{n}/|H|)$.
\end{thm}

\begin{proof}
Let $H$ be transitive of order at most $2^{\sqrt{n}/4}$. Observe
that if
\begin{equation}
\label{eeee1111}
n - \frac{n}{2\log_{2}|H|} \leq n - \log_{k}\left(2kn|H|^{2}\right),
\end{equation}
then we are done by using Lemma~\ref{eee111}. Inequality
(\ref{eeee1111}) is equivalent to the inequality $n \geq 2
(\log_{2}|H|) (\log_{k}(2kn|H|^{2}))$. Since the right-hand side is
at most $10 (\log_{2}|H|)^{2}$ and $|H| \leq 2^{\sqrt{n}/4}$,
inequality (\ref{eeee1111}) is satisfied, finishing the proof of the
theorem.
\end{proof}

We finish this section with the following.

\begin{thm}
\label{t1}
If $H$ is primitive and not isomorphic to any of the groups
\begin{enumerate}
    \item $\AAA_n$, $\SSS_n$,

    \item $\AAA_m$, $\SSS_m$ acting on the set of $2$-element subsets of $\{ 1, \ldots m \}$ with $n = \binom{m}{2}$,

    \item a group $H$ satisfying $(\AAA_{m})^{2} \leq H \leq \SSS_m \wr \SSS_2$ where $n = m^2$,
\end{enumerate}
then \[ \frac{k^{n}}{|H|} \leq k(G) < \Big(1+ \frac{1}{kn}\Big) \frac{k^n}{|H|}\] for every sufficiently large $n$ or $k$.
\end{thm}

\begin{proof}
The lower bound follows from Lemma \ref{l1}. Without the
classification of finite simple groups Sun and Wilmes
\cite[Corollary 1.6]{SW} proved that
\begin{equation}
    \label{f2}
    |H| \leq \exp \left(O(n^{1/3} \log^{7/3}n )\right),
\end{equation}
unless $H$ is a family of primitive groups appearing in (i), (ii) or
(iii) of the statement of the theorem. Since the right-hand side of
(\ref{f2}) is less than $2^{\sqrt{n}/4}$ for every sufficiently
large $n$, the result follows from Theorem \ref{t3} for every
sufficiently large $n$. When $n$ is bounded and $k \to \infty$, the
statement follows from the paragraph after (\ref{eee111}).
\end{proof}


\section{Bounding $k(X\wr \SSS_n)$ and $k(X\wr\AAA_n)$}
\label{sec:4}

The goal of this section is to prove the main result in the case
$H\in\{\SSS_n,\AAA_n\}$. This is the first exception singled out in
Theorem~\ref{t1}.

We remark that although the main theorem of \cite{GM} depends on the
classification of finite simple groups, a part of this result does
not: if $Y$ is a Young subgroup of $\SSS_n$ then both $k(Y)$ and
$k(Y \cap \AAA_n)$ are at most $5^{n/3}$. This latter statement is
used in the two proofs of Theorem \ref{thm:first-exception}
presented below.

\begin{thm}
\label{thm:first-exception} Assume the hypothesis and notation of
Theorem~\ref{t2}. If $H \in \{ \SSS_n, \AAA_n \}$, then $k(G) \leq
k^{n}$ for every sufficiently large $n$ or $k$.
\end{thm}

\begin{proof}
Let $n$ be bounded by an absolute constant. For every sufficiently
large $k$ and for any $n$ at least $3$, we have $$k(G) \leq 2 \cdot
k(X \wr \AAA_n ) \leq 2 \cdot \Big(1 + \frac{1}{kn}\Big)
\frac{k^{n}}{|\AAA_n|} < k^{n}$$ by Lemma~\ref{eee111}. When $n=2$
(and $H = \SSS_2$) the action is regular and this case was treated
earlier (see \cite[Proposition 8.5d]{S}, \cite[Lemma 4.3]{GG} or
Section 2). Let $n \geq 5$. Observe that $n(\SSS_n,\Irr(X)^n)$ is
exactly the number of $k$-tuples $(x_1,...,x_k)\in\ZZ_{\geq 0}^k$
such that $n=x_1+\cdots +x_k$. (These tuples are often referred to
as \emph{weak $k$-compositions} of $n$,  and their number is given
by ${n+k-1 \choose k-1}$.) We have
$$n(\SSS_n,\Irr(X)^n) =  {n+k-1 \choose k-1} \leq \min \Big\{
(n+1)^{k-1}, k \cdot {\Big( \frac{k+1}{2} \Big)}^{n-1} \Big\}.$$ If
$k$ is bounded by an absolute constant, then $$k(G) \leq 5^{n/3}
\cdot 2 \cdot n(\SSS_n,\Irr(X)^n) \leq 5^{n/3} \cdot 2 \cdot
(n+1)^{k-1} < k^{n},$$ for every sufficiently large $n$, by Lemma
\ref{l1} and \cite{GM}. Let $k \geq 100$. We have $$k(G) \leq
5^{n/3} \cdot 2 \cdot k \cdot {\Big( \frac{k+1}{2} \Big)}^{n-1} \leq
k^n,$$ and the proof is complete.
\end{proof}

We shall need a variation of \cite[Lemma~2.1]{DMP} in the next
section.

\begin{lem}
\label{lem:DMP} For every $\epsilon$ and $\gamma$ with
$0<\epsilon<1$ and $0<\gamma<1$, there exists $N=N(\epsilon,\gamma)$
such that for any $n\geq N$ whenever $ x \in \SSS_n$ satisfies
$(1-(1-\epsilon)\gamma)n \leq \sigma(x)$, then $|x^{\SSS_n}| < 2
\cdot n^{4} \cdot |\SSS_n|^{\gamma}$.
\end{lem}

\begin{proof}
According to \cite[Lemma~2.1]{DMP}, whenever $x \in \AAA_n$
satisfies $(1-(1-\epsilon)\gamma)n \leq \sigma(x)$, then
$|x^{\SSS_n}| \leq 2 \cdot |x^{\AAA_n}| < 2 \cdot |\AAA_n|^{\gamma}
< 2 \cdot |\SSS_n|^{\gamma}$.

Let $x \in \SSS_n \setminus \AAA_{n}$ satisfy the inequality
$(1-(1-\epsilon)\gamma)n \leq \sigma(x)$. In the disjoint cycle
decomposition of $x$ there is a cycle $\pi$ of even length, say
$2r$. Let $c_{r}$ and $c_{2r}$ be the number of cycles of lengths
$r$ and $2r$ respectively in the disjoint cycle decomposition of
$x$. We have $|\bC_{\SSS_{n}}(x)| = a \cdot r^{c_{r}} \cdot c_{r}!
\cdot (2r)^{c_{2r}} \cdot c_{2r}!$ for some positive integer $a$.
Let $y \in \AAA_n$ be the permutation obtained from $x$ by replacing
$\pi$ by $\pi^2$. We have $|\bC_{\SSS_{n}}(y)| = a \cdot r^{c_{r}+2}
\cdot (c_{r}+2)! \cdot (2r)^{c_{2r}-1} \cdot (c_{2r}-1)!$. It is
easy to see that $|\bC_{\SSS_n}(y)| \leq n^{4} \cdot
|\bC_{\SSS_{n}}(x)|$ from which it follows that $|x^{S_{n}}| \leq
n^{4} \cdot |y^{\SSS_{n}}|$. Since $\sigma(y) = \sigma(x) + 1$, we
have $(1-(1-\epsilon)\gamma)n \leq \sigma(y)$ by hypothesis and so
$|y^{\SSS_n}| < 2 \cdot |\SSS_n|^{\gamma}$ by the first paragraph.
This gives $|x^{\SSS_n}| < 2 \cdot n^{4} \cdot |\SSS_n|^{\gamma}$.
\end{proof}

Theorem \ref{thm:first-exception} can also be proved using
Lemma~\ref{lem:DMP}, as follows.

\begin{proof}[Second proof of Theorem \ref{thm:first-exception}]
Let $\epsilon$ and $\gamma$ be such that $0<\epsilon<1$ and
$0<\gamma<1$ such that $\delta := 1 -
(1-\epsilon)\gamma<1-\log_2(5)/3$. Let $\beta$ be such that $\gamma
< \beta <1$. There exists by Lemma \ref{lem:DMP} an integer $N$ such
that whenever $n \geq N$ the inequality $\delta n \leq \sigma(x)$
(for $x \in \SSS_n$) implies $|x^{\SSS_n}| < |\SSS_n|^{\beta}$. It
follows that the number of elements $x \in \SSS_n$ such that
$\sigma(x) \geq \delta n$ is less than $|\SSS_{n}|^{\beta} p(n)$
where $p(n)$ denotes the number of partitions of $n$.

By Lemma~\ref{l2} we have
$$n(\SSS_n,\Irr(X)^n) =\frac{1}{|\SSS_n|}\sum_{h\in \SSS_n}k^{\sigma(h)} =
\frac{1}{|\SSS_n|}\underset{\sigma(h) < \delta n}{\sum_{h\in \SSS_n}}k^{\sigma(h)} + \frac{1}{|\SSS_n|}
\underset{\sigma(h) \geq \delta n}{\sum_{h\in \SSS_n}}k^{\sigma(h)}$$
$$< k^{\delta n} + \frac{1}{|\SSS_n|} k^{n}  +  |\SSS_n|^{\beta -1} p(n) k^{n-1}.$$
Since $p(n)<13.01^{\sqrt{n}}$ by \cite{Erdos}, it follows that
$$n(\SSS_n,\Irr(X)^n) < \frac{k^{n}}{2\cdot 5^{n/3}}$$ for every
sufficiently large $n$ or $k$. By Lemma \ref{l1} and \cite{GM}, it
follows that $$k(G) = k(X \wr H) \leq 2 \cdot
5^{n/3}n\left(\SSS_n,\Irr(X)^n\right) < 2 \cdot
5^{n/3}\frac{k^{n}}{2\cdot 5^{n/3}} = k^{n}$$ for every sufficiently
large $n$ or $k$, as wanted.
\end{proof}

In order to finish the proof of Theorem \ref{t2} for primitive
groups, we may assume that $n \to \infty$. This follows from Theorem
\ref{thm:first-exception} and the paragraph following
(\ref{eee111}).


\section{Bounding the number of orbits of $\SSS_m$ on $\Irr(X^{{m\choose\ell}})$}
\label{sec:bounding-orbit}

To complete the proof of Theorem~\ref{t2} for primitive groups, it
remains to address the families of groups listed in (ii) and (iii)
of Theorem~\ref{t1}. These exceptional primitive permutation groups,
together with the groups in (i), fall into a broader class of
groups, which we will analyze collectively. We remark that tackling
each specific family individually does not significantly simplify
the proof.

\begin{defn}
\label{def} We say that $H$ is a \emph{large base permutation group}
of degree $n$ if $(\AAA_m)^t \leq H \leq \SSS_m \wr \SSS_t$ with $t
\geq 1$ and $m\geq 5$, where the action of $\SSS_m$ is on
$\ell$-element subsets of $\{1, \ldots , m\}$ with $1\leq\ell<m/2$
and the wreath product has the product action of degree $n = {m
\choose \ell}^t$.
\end{defn}

\begin{notation}\label{notation} We will fix the following notation
when working with large base groups:
\begin{enumerate}[\rm(i)]
\item $\Omega$ is the set of $\ell$-subsets of $\{1,\ldots,m\}$,
\item $B_t:=\Irr(X)^{{m\choose\ell}^t}$,
\item $B:=B_1=\Irr(X)^{{m\choose\ell}}$.
\end{enumerate}
\end{notation}

The goal of this section is to obtain an asymptotic bound for
$n(\SSS_m,B)$. From this point on we change notation. For a
permutation $\pi$ in $\SSS_m$, we denote the number of cycles in the
disjoint cycle decomposition of $\pi$ by $\sigma(\pi)$, while
$\sigma'(\pi)$ will denote the number of cycles in the disjoint
cycle decomposition of $\pi$ acting on the set of $\ell$-element
subsets of $\{ 1, \ldots, m \}$.

Given $j \in \{1,\ldots, m\}$, we write
$$\mathcal{S}(j,m):=|\{\pi \in \SSS_m|\sigma(\pi)=j\}|.$$
This number $\mathcal{S}(j,m)$ is often referred to as the Stirling
number of the first kind.

\begin{lem}
\label{lem:101}
$\mathcal{S}(j,m)<(m!)^{0.41}$ for every sufficiently large $m$ and $j>3m/4$.
\end{lem}

\begin{proof}
In Lemma \ref{lem:DMP}, let us take $\gamma=2/5$ and
$1-(1-\epsilon)\gamma=3/4$ (that is, $\epsilon=3/8$). Assume that
$m\geq N(\epsilon,\gamma)$ and $j>3m/4$. The set $\{\pi \in
\SSS_m|\sigma(\pi)=j\}$ is a union of conjugacy classes of $\SSS_m$.
Since all elements in $\{\pi \in \SSS_m|\sigma(\pi)=j\}$ satisfy
that $ \sigma(\pi)=j > 3m/4=(1-(1-\epsilon)\gamma)m$, we deduce that
$$\mathcal{S}(j,m)\leq 2m^4p(m)|\SSS_m|^{\gamma}<
2m^413.01^{\sqrt{m}}|\SSS_m|^{2/5},$$ and the lemma follows.
\end{proof}


\begin{lem}
\label{lem:100} For $\pi\in\SSS_m$, let $\fix(\pi)$ denote the set
of $\ell$-subsets of $\{1,\ldots m\}$ that are fixed under $\pi$.
There exists a positive integer $N$ such that \[|\fix(\pi)|<
\frac{3}{4}{m\choose \ell}\] for every $m\geq N$, $1\leq\ell < m/2$,
and $\sigma(\pi)\leq 3m/4$.
\end{lem}

\begin{rem}
We choose the constant $3/4$ in the bound to streamline the proofs.
The same argument shows that, for any positive constance $\epsilon$,
there exists $\delta>0$ such that $|\fix(\pi)|< \epsilon{m\choose
\ell}$ whenever $\sigma(\pi)\leq \delta m$. We also note that a
strong bound for $|\fix(\pi)|$ is given in \cite[Lemma~3.2]{EG},
from which our result can alternatively be deduced. We thank the
referee for pointing out this reference.
\end{rem}

\begin{proof}[Proof of Lemma~\ref{lem:100}]
For each $i$ with $1 \leq i \leq m$, let $\alpha_i$ be the number of
cycles in $\pi$ of length $i$. We have $\sigma(\pi)=\sum_i \alpha_i$
and $m=\sum_i i\alpha_i$. Let $\mathcal{P}(\ell)$ denote the set of
partitions of the integer $\ell$. For each $\lambda\in
\mathcal{P}(\ell)$, let $\lambda_i$ denote the number of parts of
$\lambda$ equal to $i$.  Then
\[
|\fix(\pi)|=\sum_{\lambda\in\mathcal{P}(\ell)} {\alpha_1\choose \lambda_1}\cdot {\alpha_2\choose \lambda_2}\cdots
\]
Using the well-known estimate ${a\choose b}\cdot{c\choose d}\leq { a+c\choose b+d}$, we obtain
\begin{equation}
\label{e10}
|\fix(\pi)|\leq \sum_{\lambda\in\mathcal{P}(\ell)} {\sum\alpha_i\choose \sum\lambda_i} = \sum_{\lambda\in\mathcal{P}(\ell)} {\sigma(\pi)\choose l(\lambda)},
\end{equation}
where $l(\lambda)$ is the number of parts of $\lambda$.

\medskip

I. Assume first that $\ell\leq \sigma(\pi)/2$. Then
${\sigma(\pi)\choose l(\lambda)}\leq {\sigma(\pi)\choose \ell}$ for
every $\lambda\in\mathcal{P}(\ell)$. It follows from \eqref{e10}
that
\[
|\fix(\pi)|\leq p(\ell) \cdot {\sigma(\pi)\choose \ell}.
\]
Assume furthermore that $m>\ell+\sigma(\pi)$. We then have
\begin{align*}
{m\choose \ell}&=\frac{m(m-1)\cdots (\sigma(\pi)+1)}{(m-\ell)(m-1-\ell)\cdots(\sigma(\pi)+1-\ell)}\cdot{\sigma(\pi)\choose\ell}\\&=\frac{m(m-1)\cdots(m-\ell+1)}{\sigma(\pi)(\sigma(\pi)-1)\cdots (\sigma(\pi)-\ell+1)}\cdot{\sigma(\pi)\choose\ell}\\
&\geq\left(\frac{4}{3}\right)^\ell\cdot{\sigma(\pi)\choose\ell},
\end{align*}
where the last inequality follows from the hypothesis on
$\sigma(\pi)$. The lemma then follows if $(4/3)^{\ell-1}> p(\ell)$.
This is true when $\ell\geq 59$, by using the bound for the
partition function in \cite{Erdos}.

We therefore may assume that $\ell\leq 58$. By \eqref{e10} and the hypothesis,
\begin{equation}
\label{e11}
|\fix(\pi)|\leq  \sum_{\lambda\in\mathcal{P}(\ell)} {[3m/4]\choose l(\lambda)}.
\end{equation}
Thus, we are done if
\[
\sum_{\lambda\in\mathcal{P}(\ell)} {[3m/4]\choose l(\lambda)}<\frac{3}{4}{m\choose \ell}.
\]
For each fixed $\ell$ with $\ell \leq 58$, we observe that the
right-hand side is a polynomial (in $m$) of degree $\ell$, while the
left-hand side is a polynomial of degree at most $\ell$ and if equal
to $\ell$ then with smaller leading coefficient. Therefore the
inequality holds for every sufficiently large $m$, and we are done.

Now assume that $m\leq \ell+\sigma(\pi)$. Then
\begin{align*}
{m\choose \ell}\geq\left(\frac{m}{m-\ell}\right)^{m-\sigma(\pi)}\cdot{\sigma(\pi)\choose\ell}\geq \left(\frac{m}{m-\ell}\right)^{m/4}\cdot{\sigma(\pi)\choose\ell}\geq \left(\frac{4}{3}\right)^{m/4}\cdot{\sigma(\pi)\choose\ell}.
\end{align*}
As above, the desired inequality follows from these bounds for every
sufficiently large $m$.

\medskip

II. Next we consider the case $\ell> \sigma(\pi)/2$. Then
${\sigma(\pi)\choose l(\lambda)}\leq {\sigma(\pi)\choose
[\sigma(\pi)/2]}$ for every $\lambda\in\mathcal{P}(\ell)$. As in the
previous case, it follows from \eqref{e10} that
\[
|\fix(\pi)|\leq p(\ell) \cdot {\sigma(\pi)\choose [\sigma(\pi)/2]}.
\]
On the other hand, by the hypothesis on $\sigma(\pi)$ and $\ell$, we have
\begin{align*}
{m\choose \ell}&=\frac{(m-[\sigma(\pi)/2])\cdots(m-\ell+1)}{\ell(\ell-1)\cdots([\sigma(\pi)/2]+1)}\cdot{m\choose[\sigma(\pi)/2]}\\
&\geq \left(\frac{5}{4}\right)^{\ell-[\sigma(\pi)/2]}\cdot{m\choose[\sigma(\pi)/2]}.
\end{align*}
The lemma now follows by similar estimates as in the previous case,
but for ${m\choose[\sigma(\pi)/2]}$ instead of ${m\choose \ell}$.
\end{proof}


\begin{prop}
\label{prop:11} Let $X$ be a nontrivial finite group with $k$
conjugacy classes. There exists a positive integer $N$ (independent
of $k$) such that
\[
n(\SSS_m,B)< 2\max\left\{k^{\frac{7}{8}{m\choose\ell}},(m!)^{-0.58}k^{{m\choose\ell}}\right\}.
\]
for every $m\geq N$ and $1\leq\ell<m/2$.
\end{prop}

\begin{proof}
Recall from Notation~\ref{notation} that $\Omega$ is the set of
$\ell$-subsets of $\{1, \ldots, m\}$. Let $\pi\in \SSS_m$. Let
$\fix(\pi)$ denote the set of $\ell$-subsets fixed by $\pi$, as in
Lemma~\ref{lem:100}. Note that
\begin{equation}
\label{e14}
\sigma'(\pi)\leq |\fix(\pi)|+\frac{1}{2}\left({m\choose\ell}-|\fix(\pi)|\right)=\frac{1}{2}\left({m\choose\ell}+|\fix(\pi)|\right).
\end{equation}
Furthermore, by Lemma~\ref{l2},
\[
n(\SSS_m,B)=\frac{1}{|\SSS_m|} \sum_{\pi\in \SSS_m} {k^{\sigma'(\pi)}}.
\]
We decompose this into two smaller sums, depending on whether
$\sigma(\pi)$  is smaller or larger than $3m/4$:
\[n(\SSS_m,B)=n_1(\SSS_m,B_1)+n_2(\SSS_m,B_1),\] where
\[
n_1(\SSS_m,B)=\frac{1}{|\SSS_m|} \sum_{\sigma(\pi)\leq 3m/4} {k^{\sigma'(\pi)}} \text{ and } n_2(\SSS_m,B)=\frac{1}{|\SSS_m|} \sum_{\sigma(\pi)> 3m/4} {k^{\sigma'(\pi)}}.
\]
By using \eqref{e14}, we get
\[
n_1(\SSS_m,B)\leq \frac{1}{m!}\sum_{\sigma(\pi)\leq 3m/4}{k^{\frac{1}{2}\left({m\choose\ell}+|\fix(\pi)|\right)}},
\]
and it follows from Lemma~\ref{lem:100} that
\[
n_1(\SSS_m,B)\leq k^{\frac{7}{8}{m\choose\ell}}
\]
for every $m\geq N_1$ for some positive integer $N_1$.

We now work on $n_2(\SSS_m,B)$. Recall that $\mathcal{S}(j,m)$
denotes the number of elements of $\SSS_m$ with precisely $j$
cycles. So
\[
n_2(\SSS_m,B) \leq \frac{1}{m!}\sum_{j> 3m/4}\mathcal{S}(j,m)k^{{m \choose \ell}}.
\]
Using the bound for $\mathcal{S}(j,m)$ in Lemma~\ref{lem:101}, we deduce that
\[
n_2(\SSS_m,B)\leq\frac{1}{4}m(m!)^{-0.59}k^{m\choose\ell}
\]
for every $m\geq N_2$ for some positive integer $N_2$.
Now taking $N_3:=\max\{N_1,N_2\}$, we arrive at
\[
n(\SSS_m,B)\leq k^{\frac{7}{8}{m\choose\ell}}+\frac{1}{4}m(m!)^{-0.59}k^{m\choose\ell}\]
for every $m\geq N_3$, and the result readily follows.
\end{proof}

%


\section{Large base groups}
\label{sec:6}

In this section we complete the proof of Theorem~\ref{t2} for
primitive permutation groups by proving the following.

\begin{thm}
\label{thm:10} Let $X$ be a non-trivial finite group and let
$k:=k(X)$. Let $t\geq 1$, $m\geq 5$, $1\leq\ell<m/2$, and
$(t,\ell)\neq(1,1)$. Let $H$ be a large base primitive permutation
group of degree $n:={m\choose\ell}^t$, as in Definition~\ref{def}.
Let $G=X\wr H$. Then $k(G)\leq k^n$ for every sufficiently large $k$
or $n$.
\end{thm}

Recall that $\Omega$ denotes the set of $\ell$-subsets of
$\{1,...,m\}$. For $\pi\in \SSS_{m}$, let $\sigma'(\pi)$ be the
number of its cycles as a permutation on $\Omega$.  For $x\in \SSS_m
\wr \SSS_t$, let $\gamma(x)$ be the number of its cycles as a
permutation on $\Omega^t$. (Of course, $\gamma(x)=\sigma'(x)$ when
$t=1$.) Recall also that $(\SSS_m)^t$ has a natural product action
on $B_t=\Irr(X)^{{m\choose\ell}^t}$ and $n((\SSS_m)^t,B_t)$ denotes
the number of its orbits.

In the next proposition we will need a general bound $f(n)$ for the
number of conjugacy classes of a permutation group of degree $n$
which does not rely on the classification of finite simple groups.
Kov\'acs and Robinson \cite[Theorem 1.2]{KR} proved that $f(n)$ may
be taken to be $5^{n-1}$.

\begin{prop}
\label{prop:13}
Assume the notation and hypothesis of Theorem~\ref{thm:10}. Then
\[
k(G)<5^{mt}\left(2^tn((\SSS_m)^t,B_t)+k^{2n/3}\right).
\]
\end{prop}

\begin{proof}
Let $D:= (\SSS_m)^t\cap H$ be the `diagonal' subgroup of $H$. By
Lemma~\ref{l2}, the number of orbits of $H$ acting on $B_{t}
=\Irr(X^n)$ is
\begin{equation}
\label{e13}
n(H,B_{t})=\frac{1}{|H|}\sum_{x\in H}k^{\gamma(x)}=\frac{1}{|H|}\sum_{x\in D}k^{\gamma(x)}+\frac{1}{|H|}\sum_{x\in H\setminus D}k^{\gamma(x)}.
\end{equation}

For $x \in H$,  we write $\fix(x)$ to denote the set of elements in
$\Omega^t$ fixed by $x$.  By (\ref{e3}) we have
$$\gamma(x)\leq \frac{n}{2}(1+ \fpr(x)),$$
where $\fpr(x)=|\fix(x)|/n$ is the fixed point ratio of $x$. Let $x \in H\setminus D$ and let us write $x=(x_1,\ldots,x_t)\pi\in H$ for $x_1,\ldots, x_t\in \SSS_m$ and $1\neq \pi\in \SSS_t$. We know that $\pi$ must contain a cycle of length $r$ for some $r \geq 2$. A straightforward computation shows that $\fix(x)\leq |\Omega|^{t-(r-1)}$ (see also \cite[Proposition 6.1]{BG} for a similar argument) and thus $$\fpr(x)\leq |\Omega|^{1-r}\leq (r+1)^{-1} \leq 1/3,$$ where the second inequality holds since $|\Omega|\geq 5$.   Thus, the second term in the
far-right-hand-side sum in (\ref{e13}) is bounded by $k^{2n/3}$.


On the other hand, for the first term, we have
\[
\frac{1}{|H|}\sum_{x\in D}k^{\gamma(x)}\leq \frac{1}{|\AAA_m|^t}\sum_{x\in D}k^{\gamma(x)}\leq \frac{1}{|\AAA_m|^t}\sum_{x\in (\SSS_m)^t}k^{\gamma(x)}=2^tn((\SSS_m)^t,B_t).
\]
We have shown that
\[
n(H,B_{t})\leq 2^tn((\SSS_m)^t,B_t)+k^{2n/3}.
\]
Note that $H\leq  \SSS_m\wr \SSS_t$ and $\SSS_m\wr \SSS_t$ may be
viewed as a subgroup of $\SSS_{mt}$. It follows that every subgroup
of $H$ has at most $5^{mt}$ classes, by \cite[Theorem 1.2]{KR}. The
desired bound now follows by using Lemma~\ref{l1}.
\end{proof}

The next lemma relates the number of orbits of the product action of
$(\SSS_m)^t$ (on $B_t$) and that of $\SSS_m$ (on $B_1$). This allows
us to use the results in Section~\ref{sec:bounding-orbit} on
bounding $n(\SSS_m,B_1)$ to obtain similar bounds for
$n((\SSS_m)^t,B_t)$, which in turn provides corresponding bounds for
$k(G)$ by using Proposition~\ref{prop:13}.

\begin{lem}
\label{lem:15}
$n((\SSS_m)^t,B_t)=n(\SSS_m,B_1)^t$.
\end{lem}

\begin{proof}
Observe that $B_t=(B_1)^t$ and an element
$x=(x_1,\ldots,x_t)\in(\SSS_m)^t$ fixes $(\chi_1,\ldots,\chi_t)\in
B_t$ if and only if each $x_i\in\SSS_m$ fixes $\chi_i\in B_1$ for
every $i$. Now,
\begin{align*}
n((\SSS_m)^t,B_t)&=\frac{1}{|\SSS_m|^t} \sum_{x\in(\SSS_m)^t} |\fix(x,B_t)|\\
&=\frac{1}{|\SSS_m|^t} \sum_{x_1\in\SSS_m}\cdots \sum_{x_t\in\SSS_m}|\fix(x_1,B_1)|\cdots |\fix(x_t,B_1)|\\
&=\frac{1}{|\SSS_m|^t}
\left(\sum_{x_1\in\SSS_m}|\fix(x_1,B_1)|\right)^t\\
&=n(\SSS_m,B_1)^t,
\end{align*}
and the lemma follows.
\end{proof}


We are now ready to prove the main result of this section.

\begin{proof}[Proof of Theorem~\ref{thm:10}] By Proposition~\ref{prop:13} and Lemma~\ref{lem:15}, we have
\begin{equation}
\label{eq:30}
k(G)<5^{mt}\left(2^tn(\SSS_m,B_1)^t+k^{2n/3}\right).
\end{equation}
Obviously, $n(\SSS_m,B_1)\leq |B_1|=k^{{m\choose\ell}}$. Hence
\[
k(G)<5^{mt}2^tk^{{m\choose\ell}t} + 5^{mt}k^{2n/3}.
\]
It is straightforward to see that, as $(t,\ell)\neq (1,1)$, both
terms on the right-hand side are less than
$\frac{1}{2}k^{{m\choose\ell}^t}$ for every sufficiently large $t$.
We assume from now on that $t$ is bounded.

Next, using Proposition~\ref{prop:11} together with \eqref{eq:30},
we have that there exists a positive integer $N$ such that, for
every $m\geq N$,
\begin{align*}
k(G)&<5^{mt}4^t(\max\{k^{\frac{7}{8}{m\choose\ell}},(m!)^{-0.58}k^{{m\choose\ell}}\})^t+5^{mt}k^{2n/3}\\
&\leq 5^{mt}4^tk^{\frac{7t}{8}{m\choose\ell} t} + 5^{mt}4^t(m!)^{-0.58t}k^{{m\choose\ell} t}+5^{mt}k^{2n/3}.
\end{align*}
With $t$ being bounded, each of these three terms is less than
$\frac{1}{3}k^{{m\choose\ell}^t}$, and therefore $k(G)\leq k^n$, for
every sufficiently large $m$.

We now assume that both $t$ and $m$ are bounded, or equivalently,
that $n$ is bounded. Given the hypothesis that $H$ is a primitive
group that is different from $\SSS_n$, it follows that $H$ does not
contain a transposition. In this case, the remark following
\eqref{eee111} shows that $k(G)<k^n$ for all sufficiently large $k$.
This completes the proof.
\end{proof}

For primitive groups $H$, Theorem~\ref{t2} follows from Theorems
\ref{t1}, \ref{thm:first-exception} and \ref{thm:10}.


\section{Semiprimitive groups}
\label{sec:semi}

In this section we complete the proof of Theorem~\ref{t2} by proving
it for semiprimitive groups which are not primitive.

\begin{proof}[Proof of Theorem~\ref{t2}]
Let $H$ be a semiprimitive permutation group. This is a transitive
permutation group all of whose normal subgroups are transitive or
semiregular. We may assume at this point that $H$ is not primitive.
The group $H$ acts on the set $\Omega$ of factors of $X^n$. Let \[Y
:= X^{n/r}\] for some divisor $r\leq n/2$ of $n$ such that $H$ acts
primitively on the set $\overline{\Omega}$ of factors of $Y^r$. Let
the kernel of this action be $K$. Since this is an intransitive
normal subgroup of $H$, it must be semiregular on $\Omega$.

Let $h$ be an element of $H$. Let  the number of cycles of $h$
acting on $\Omega$ and $\overline{\Omega}$ be denoted by
$\sigma_{\Omega}(h)$ and $\sigma_{\overline{\Omega}}(h)$,
respectively. Observe that $\sigma_{\Omega}(h) \leq (n/r) \cdot
\sigma_{\overline{\Omega}}(h)$.

We have
\begin{align*}
\alpha(H) := \max_{1 \not= h \in H} \frac{\sigma_{\Omega}(h)}{n} &= \max \left\{ \max_{1 \not= h \in K} \frac{\sigma_{\Omega}(h)}{n}, \max_{h \in H \setminus K} \frac{\sigma_{\Omega}(h)}{n}\right\}\\
& \leq \max \left\{ \frac{1}{2}, \max_{h \in H \setminus K} \frac{\sigma_{\overline{\Omega}}(h)}{r}\right\}.
\end{align*}
It follows that if $r$ is bounded by an absolute constant, then
$\alpha(H)$ is a fixed number less than $1$ and so $k(G) \leq k^n$
for every sufficiently large $n$ or $k$ by (\ref{ee11}). We may
therefore assume that $r \to \infty$, in particular, $n \to \infty$.

If $H/K$ is not a large base group (see Definition~\ref{def}), then
$$|H| = |H/K| |K| \leq |H/K| \cdot n = \exp(O(n^{1/3}\log^{7/3}n))$$
by (\ref{f2}). In this case the result follows from Theorem
\ref{t3}.

For general $H/K$, we have $$n(H, \Irr(X^{n})) = \frac{1}{|H|}
\sum_{h \in H} k^{\sigma_{\Omega}(h)} = \frac{1}{|H|} \Big( \sum_{h
\in H \setminus K}  k^{\sigma_{\Omega}(h)} + \sum_{h \in K}
k^{\sigma_{\Omega}(h)}  \Big) \leq$$
$$\leq \frac{1}{|H|} \Big( |K| \sum_{1 \not= \bar{h} \in H/K} k^{\sigma_{\overline{\Omega}}(\bar{h}) \cdot (n/r)} + k^{n} + (|K|-1) k^{n/2}  \Big) <$$
$$< n\left(H/K, \Irr(Y^{r})\right) + \frac{k^{n}}{|H|} + \frac{n k^{n/2}}{|H|},$$
where $n(H/K, \Irr(Y^{r}))$ is the number of orbits of $R = H/K$ on
$\Irr(Y^{r}) = {\Irr(Y)}^{r}$ where the action of $R$ is defined from the
action of the primitive permutation group $R$ acting on the set of factors
of $Y^{r}$ (that is, $R$ acts primitively on the set of factors of
${\Irr(Y)}^{r}$).

The number $k(G)$ is equal to the sum of the numbers of conjugacy
classes of $n(H, \Irr(X^{n}))$ inertia subgroups, by Lemma \ref{l1}.
As before, let $e$ be the maximum of these numbers. Since $K$ is
semiregular, at most $\sum_{1 \not= h \in K} k^{\sigma_{\Omega}(h)}
< (n/r) k^{n/2}$ of the inertia subgroups intersect $K$
nontrivially. These numbers contribute less than $e (n/r) k^{n/2}
\leq (n/r)^{2} \cdot 5^{r/3} k^{n/2}$ to $k(G)$. (For $H/K \in \{
\SSS_r, \AAA_r \}$ this follows from  \cite{Nagao} and from the
statement in the second paragraph of Section 3, otherwise $|H/K|
\leq 5^{r/3}$ for every sufficiently large $r$ by \cite[Corollary
1.6]{SW}.) Since $r \leq n/2$, we have $(n/r)^{2} \cdot 5^{r/3}
k^{n/2} \leq n^{2} \cdot 5^{n/6} k^{n/2}$ and this is less than
$k^{n}/16$ for every sufficiently large $n$. Thus we have
\begin{align*}
k(G) &\leq e_K \cdot n\left(H, \Irr(X^{n})\right) + \frac{k^{n}}{16}\\
 &< e_K \cdot n\left(H/K, \Irr(Y^{r})\right) + \frac{e_K \cdot k^{n}}{|H|} + \frac{e_K \cdot n k^{n/2}}{|H|} + \frac{k^{n}}{16},
\end{align*}
where $e_K$ denotes the maximum of the numbers of conjugacy classes
of those inertia subgroups of $H$ which intersect with $K$
trivially. Note that $e_K$ is at most the maximum of the numbers of
classes of subgroups of $H/K$.

Recall that we are done when $H/K$ is not a large base group, and so
we assume in the remainder of the proof that $H/K$ is a large base
group. We shall follow the notation in Definition~\ref{def} and
Notation~\ref{notation}, with $Y$ and $r$ in place of $X$ and $n$,
respectively. In particular, $(\AAA_m)^t\leq H/K\leq \SSS_m \wr
\SSS_t$ for some $t\geq 1$ and $m\geq 5$. Also,
$r={m\choose\ell}^t$.

Since $H$ is not abelian, we have $e_K \leq (5/8) |H|$ by
\cite{Gustafson}. It follows that $$\frac{e_K \cdot k^{n}}{|H|} +
\frac{e_K \cdot n k^{n/2}}{|H|} + \frac{k^{n}}{16} \leq \frac{3}{4}
k^{n}$$ for every sufficiently large $n$. Note that $H/K$ can be
viewed as a subgroup of $\SSS_{mt}$, and so $e_K \leq 5^{mt}$ by
\cite[Theorem 1.2]{KR}. To establish $k(G)\leq k^n$ for sufficiently
large $n$ or $k$, it is now sufficient to show that
\begin{equation}
\label{eq:1111}
n(H/K, \Irr\left(Y^{r})\right) \leq \frac{k^{n}}{4\cdot 5^{mt}}
\end{equation}
for every sufficiently large $r$, or equivalently, every
sufficiently large $m$ or $t$.

Let $(t,\ell) = (1,1)$. In this case we may replace the above bound
$e_K \leq 5^{mt}$ by $e_K \leq 5^{mt/3}$ as discussed above. We have
$$n(H/K, \Irr\left(Y^{r})\right) \leq 2 \cdot \binom{r + k^{n/r} - 1}{r} \leq 2 \cdot 3^{r} \Big(\frac{r + k^{n/r} -1}{r}\Big)^{r} < 2 \cdot 3^{r} \Big(\frac{k^{n/r}}{r} + 1 \Big)^{r}.$$ If $k^{n/r}/r \geq 1$, then $$n(H/K, \Irr\left(Y^{r})\right) \leq 2 \cdot 6^{r} \Big( \frac{k^{n}}{r^r} \Big) < k^{n}/(4\cdot 5^{mt/3})$$ for every sufficiently large $r = m$. Let $k^{n/r} \leq r$. Then $n(H/K, \Irr\left(Y^{r})\right) \leq 4^{r}$. This is less than $k^{n}/(4\cdot 5^{r/3})$ for every sufficiently large $r$, unless $n = 2r$ and $k=2$. In the exceptional case $n(H/K, \Irr\left(Y^{r})\right) \leq (r+3)(r+2)(r+1)/3$, which is again less than $k^{n}/(4\cdot 5^{r/3})$ for every sufficiently large $r$. Let $(t,\ell) \not= (1,1)$.

First, arguing as in the proof of Proposition~\ref{prop:13} and
using Lemma~\ref{lem:15}, we have
\[
n(H/K, \Irr\left(Y^{r})\right)\leq 2^tn\left((\SSS_m)^t,B_t\right)+k(Y)^{2r/3}=2^tn\left(\SSS_m,B_1\right)^t+k(Y)^{2r/3}.
\]
When $t\to\infty$, one may use the obvious bound
$n\left(\SSS_m,B_1\right)\leq |B_1|=k(Y)^{m\choose\ell}$ to achieve
\eqref{eq:1111}. So we assume that $t$ is bounded.

Next, using Proposition~\ref{prop:11}, we deduce that
\begin{align*}
n(H/K, \Irr\left(Y^{r})\right)
\leq 4^t k(Y)^{\frac{7}{8}{m\choose\ell}t} + 4^t(m!)^{-0.58t} k(Y)^{{m\choose\ell}t} + k(Y)^{2r/3}.
\end{align*}
for every sufficiently large $m$.
As $k(Y)=k^{n/r}$, it follows that
\[
n(H/K, \Irr\left(Y^{r})\right) \leq 4^t k^{\frac{7}{8r}{m\choose\ell}tn} + 4^t(m!)^{-0.58t} k^{\frac{1}{r}{m\choose\ell}tn} + k^{2n/3}.
\]
With $n\geq 2r=2{m\choose\ell}^t$, $m\to\infty$, and $t$ being
bounded, it is straightforward to verify that this sum is less than
the right-hand side of \eqref{eq:1111}, and the proof is complete.
\end{proof}

\bigskip

\centerline{\bf Acknowledgement}

\bigskip

This work was carried out while the first and third authors were
visiting Budapest. They gratefully acknowledge the hospitality of
the Alfr\'ed R\'enyi Institute of Mathematics.

\end{document}